\documentclass[a4paper,reqno]{amsart}
\usepackage{amssymb, amsmath, amscd}
\def\dvol{d\nu}
\def\gronko{\vphantom{\vrule height 11pt}}
\def\pone{\hbox{\rm\textsf{1}\hspace*{-0.9ex}
  \rule{0.15ex}{1.3ex}\hspace*{0.9ex}}}
  \newtheorem{theorem}{Theorem}[section]

\newtheorem{lemma}[theorem]{Lemma}

\newtheorem{remark}[theorem]{Remark}

\def\Id{\operatorname{Id}}
\def\mapright#1{\smash{\mathop{\longrightarrow}\limits\sp{#1}}}
\makeatletter
 \@addtoreset{equation}{section}
\makeatother
\begin{document}
\title[Homogeneous K\"ahler--Weyl Structures]{Homogeneous 
4-dimensional\\ K\"ahler--Weyl Structures}
\author[Brozos-V\'{a}zquez et. al.]
{M. Brozos-V\'{a}zquez, E. Garc\'{\i}a-R\'{\i}o, P. Gilkey,\\
 and R. V\'{a}zquez-Lorenzo}
\address{MB: Department of Mathematics, University of A Coru\~na, Spain}
\email{mbrozos@udc.es}
\address{EG and RV: Faculty of Mathematics, University of Santiago de Compostela, Spain}
\email{eduardo.garcia.rio@usc.es and ravazlor@edu.xunta.es}
\address{PG: Mathematics Department, University of Oregon,
  Eugene OR 97403, USA}
  \email{gilkey@uoregon.edu}
\begin{abstract}{Any pseudo-Hermitian or para-Hermitian manifold of dimension
$4$ admits a unique K\"ahler--Weyl structure; this structure is locally
conformally K\"ahler if and only if the alternating Ricci tensor $\rho_a$
vanishes. The tensor $\rho_a$ takes values in a certain representation
space. In this paper, we show that any algebraic possibility
$\Xi$ in this representation space can in fact be geometrically realized by a left-invariant 
K\"ahler--Weyl structure on a $4$-dimensional Lie group in either the pseudo-Hermitian or the para-Hermitian setting.
}\end{abstract}
\maketitle
\section{Introduction}
Let $M$ be a smooth manifold of dimension $m=2\bar m\ge4$ with $H^1(M;\mathbb{R})=0$; 
we are only really interested in local theory so
this cohomology vanishing condition poses no real restriction. Let
$\nabla$ be a torsion free connection on the tangent bundle of
$M$, and let $g$ be a pseudo-Riemannian metric on $M$ of signature $(p,q)$.

\subsection{Weyl structures} The triple $(M,g,\nabla)$ is said to be a {\it Weyl}
structure and $\nabla$ is said to be a {\it Weyl connection} if there exists a smooth $1$-form $\phi$ on $M$ so that
$$\nabla g=-2\phi\otimes g\,.$$
Weyl \cite{W22} used these geometries in an attempt to unify gravity
with electromagnetism -- although this approach failed for physical reasons, the resulting geometries are still of importance and 
there is a vast literature on the subject. See, for example, \cite{BM11, GNS11, HN11, K11,  L04, M11}; note
that the indefinite signature setting is of particular importance \cite{DGST11, GJ08, L04, MOP12} as is the complex
setting \cite{KT12, LM09, OV08}. The field is a vast one and we only cite a few
representative recent examples.
 We introduce the following notational conventions and follow the treatment of \cite{B12} (Section 6.5) 
 which is based on work of \cite{GI94,H32,PS91,PT93}.
Let $\mathcal{R}$ be the curvature operator of a Weyl connection $\nabla$ and let $R$ be the associated curvature tensor:
$$
\mathcal{R}(x,y)z:=(\nabla_x\nabla_y-\nabla_y\nabla_x-\nabla_{[x,y]})z\text{ and }
R(x,y,z,w):=g(\mathcal{R}(x,y)z,w)\,.
$$
Let $\rho$ be the {\it Ricci tensor}, let $\rho_a$ be the {\it alternating Ricci tensor}, and let $\rho_s$ be the {\it symmetric Ricci tensor}:
\begin{eqnarray*}
&&\rho(x,y):=\operatorname{Tr}\{z\rightarrow\mathcal{R}(z,x)y\},\qquad
\rho_a(x,y)=\textstyle\frac12(\rho(x,y)-\rho(y,x)),\\
&&\rho_s(x,y)=\textstyle\frac12(\rho(x,y)+\rho(y,x))\,.
\end{eqnarray*}
We have the symmetries:
\begin{eqnarray}
&&R(x,y,z,w)+R(y,x,z,w)=0,\nonumber\\
&&R(x,y,z,w)+R(y,z,x,w)+R(z,x,y,w)=0,\label{eqn-1.a}\\
&&R(x,y,z,w)+R(x,y,w,z)=-\textstyle\frac4m\rho_a(x,y)g(z,w)\nonumber\,.
\end{eqnarray}

This is a conformal theory; if $\tilde g=e^{2f}g$ is a conformally equivalent metric, 
then $(M,\tilde g,\nabla)$ is again a Weyl
structure with associated $1$-form $\tilde\phi:=\phi-df$. The Weyl structure is said to be 
{\it trivial} (or {\it integrable}) if $\nabla$ is the
Levi-Civita connection of a conformally equivalent metric or, equivalently (given our 
assumption that $H^1(M;\mathbb{R})=0$), if $d\phi=0$. The following formula
shows that the Weyl structure is trivial if and only if $\rho_a=0$:
\begin{equation}\label{eqn-1.b}
d\phi=-\textstyle\frac2m\rho_a\,.
\end{equation}

\subsection{Complex manifolds} Let $J_-$ be an almost complex structure on $M$,
 i.e. an automorphism of the tangent bundle $TM$ so that $J_-^2=-\Id$. 
 We say that $J_-$ is {\it integrable} if there exist local coordinates 
 $(x_1,...,x_{\bar m},y_1,...,y_{\bar m})$ on a neighborhood of any point of $M$ so that
 $$J_-\partial_{x_i}=\partial_{y_i}\quad\text{and}\quad J_-\partial_{y_i}=-\partial_{x_i}\,.$$
We define the {\it Nijenhuis
tensor} by setting:
$$
N_{J_-}(x,y):=[x,y]+J_-[J_-x,y]+J_-[x,J_-y]-[J_-x,J_-y]\,.
$$
Then $J_-$ is integrable if and only if $N_{J_-}$ vanishes. Let 
$T_{\mathbb{C}}M:=TM\otimes_{\mathbb{R}}\mathbb{C}$ be the complexified tangent bundle.
Let $\mathcal{W}_\pm$ be the $\pm\sqrt{-1}$ eigenbundles  of $J_-$:
\begin{eqnarray}\label{eqn-1.c}
   &&\mathcal{W}_\pm:=\left\{Z\in T_{\mathbb{C}}M:J_-Z=\pm\sqrt{-1}Z\right\}\\
   &&\qquad\phantom{}=\mathbb{C}\cdot\left\{E\mp\sqrt{-1}J_-E\right\}_{E\in TM}\,.\nonumber
\end{eqnarray}
The distribution $\mathcal{W}_-$ (or, equivalently, $\mathcal{W}_+={\bar{\mathcal{W}}}_-$) 
determines the almost complex structure $J_-$. Furthermore, $J_-$ is integrable if and only if
the complex Frobenius condition is satisfied:
$$[C^\infty(\mathcal{W}_-),C^\infty(\mathcal{W}_-)]\subset C^\infty(\mathcal{W}_-)\,.$$

\subsection{Para-complex manifolds} Let $J_+$ be an almost para-complex structure on $M$, i.e. 
an automorphism  $J_+$ of $TM$ so that $J_+^2=\operatorname{Id}$ and so that 
$\operatorname{Tr}(J_+)=0$. We say that $J_+$ is {\it integrable} if there exist local coordinates 
$(x_1,...,x_{\bar m},y_1,...,y_{\bar  m})$  on a neighborhood of any point of $M$  
so that
$$J_+\partial_{x_i}=\partial_{y_i}\quad\text{and}\quad J_+\partial_{y_i}=\partial_{x_i}\quad\text{for}\quad 1\le i\le \bar m\,.$$
We form the real eigenbundles 
$$\mathcal{W}_\pm:=\left\{X\in C^\infty(TM):J_+X=\pm X\right\}=\{E\pm J_+E\}_{E\in TM}\,.$$
Then $J_+$ is integrable if and only if $[C^\infty(\mathcal{W}_\pm),C^\infty(\mathcal{W}_\pm)]
\subset C^\infty(\mathcal{W}_\pm)$; 
in contrast to the complex setting, both conditions are required. 
The Nijenhuis tensor in this context is defined by
$$
N_{J_+}(x,y):=[x,y]-J_+ [J_+ x,y]-J_+ [x,J_+ y]+[J_+ x,J_+ y]\,;
$$
the two distributions $\{\mathcal{W}_+,\mathcal{W}_-\}$ determine $J_+$ and the para-complex structure 
is integrable if and only if $N_{J_+}=0$.
\subsection{(Para)-K\"ahler--Weyl structures} Let $g$ be a pseudo-Riemannian metric on $M$ of
signature $(p,q)$. 
In the complex setting, we assume that $J_-$ is almost
pseudo-Hermitian (i.e. $J_-g=g$) and in the para-complex setting, we assume that $J_+$ is 
almost para-Hermitian (i.e. $J_+g=-g$); here, we extend $J_\pm$ to act naturally on tensors of all types. 
We shall use the notation $J_\pm$ as it is a convenient formalism for discussing both geometries
in a parallel format; we shall never, however, be considering both the complex setting ($-$) and
the para-complex setting $(+)$ at the same moment.

Let $(M,g,\nabla)$ be a Weyl structure and let $(M,g,J_\pm)$ be a pseudo-Hermitian ($-$)
or a para-Hermitian ($+$) structure. We say that the quadruple $(M,g,\nabla,J_\pm)$ is 
a  {\it K\"ahler--Weyl} structure if $\nabla J_\pm=0$; this
necessarily implies that $J_\pm$ is integrable so we restrict to this setting henceforth. 
We then have the additional curvature
symmetry:
\begin{equation}\label{eqn-1.d}
R(x,y,J_\pm z,J_\pm w)=\mp R(x,y,z,w)\,.
\end{equation}
We say that the K\"ahler--Weyl structure is {\it trivial} (or {\it integrable}) if $\nabla$
 is the Levi-Civita connection of some conformally equivalent K\"ahler metric.
As the K\"ahler--Weyl structure is trivial if and only if the alternating Ricci tensor
$\rho_a=0$, attention is focused on this tensor. 
Pedersen, Poon, and Swann
\cite{PPS93} used work of Vaisman \cite{V82,V83}
 to establish the following result
 in the positive definite setting; the
extension to the pseudo-Hermitian or to the para-Hermitian setting is immediate (see, for example, the discussion in \cite{B12}):
\begin{theorem} If $m>4$, then any K\"ahler--Weyl structure is trivial.
\end{theorem}

Thus only dimension $m=4$ is interesting in this theory. Let
$$\Omega(x,y):=g(x,J_\pm  y)$$
denote the K\"ahler form. Let $\delta$ be the co-derivative. If $\star$ is the {\it Hodge operator} and if
$\Omega_{ij;k}$ are the components of $\nabla^g\Omega$, then the {\it Lee form} $\delta\Omega$
is given by:
$$\delta\Omega=-\star d\star\Omega=g^{ij}\Omega_{ij;k}dx^k\,.$$
Note that $J_\pm\delta\Omega$ is called the {\it anti-Lee form}. We refer to Section~\ref{sect-3.1} for further details.
The following result was established \cite{KK10} in
the Riemannian setting; the proof extends without change to this more general context  -- we also refer to \cite{GN11c,GN12a}
for another treatment and to \cite{LH10} for related material. 

\begin{theorem}\label{thm-1.2}
Every para-Hermitian ($+$) or pseudo-Hermitian ($-$) manifold  of dimension $4$ admits a unique 
K\"ahler--Weyl structure where we have that
$\phi=\pm\textstyle\frac12J_\pm\delta\Omega$,
$\nabla_xy:=\nabla_x^gy+\phi(x)y+\phi(y)x-g(x,y)\phi^\star$ (here $\phi^\star$ is the dual vector field), and
$\rho_a=\mp dJ_\pm\delta\Omega=-2d\phi$.
\end{theorem}

\subsection{The (para)-unitary group}
We study the quadruple $(T_PM,g_P,J_{\pm,P},R_P)$ where $P$ is a point of a K\"ahler--Weyl manifold. We shall
eventually be interested in the homogeneous setting and thus the point $P$ will be inessential.
We pass to the algebraic setting
and work abstractly.  Let
$(V,\langle\cdot,\cdot\rangle,J_\pm)$ be a pseudo-Hermitian ($-$) or a para-Hermitian ($+$) vector space. 
Introduce the following structure groups:
\begin{eqnarray*}
&&\mathcal{O}=\mathcal{O}(V,\langle\cdot,\cdot\rangle):=\{T\in\operatorname{GL}(V):T^*\langle\cdot,\cdot\rangle=
        \langle\cdot,\cdot\rangle\},\\
&&\mathcal{U}=\mathcal{U}(V,\langle\cdot,\cdot\rangle,J_\pm):=\{T\in\mathcal{O}:TJ_\pm=J_\pm T\},\\
&&\mathcal{U}^\star=\mathcal{U}^\star(V,\langle\cdot,\cdot\rangle,J_\pm):=\{T\in\mathcal{O}:TJ_\pm= J_\pm T
\text{ or }TJ_\pm=-J_\pm T\}\,.
\end{eqnarray*}
There is a natural $\mathbb{Z}_2$ valued character $\chi$ of $\mathcal{U}^\star$ so that
$$TJ_\pm=\chi(T)\cdot J_\pm T\quad\text{for}\quad T\in\mathcal{U}^\star\,.$$
We let $\mathcal{O}_0$ and $\mathcal{U}_0$ denote the connected component of the identity.

These groups act on tensors of all types. Let $\pone$ be the trivial $\mathcal{U}^\star$ module
and, by an abuse of notation, let $\chi$ be the linear representation space corresponding to the character given above. 
We define the following modules:
$$
\begin{array}{ll}
S_{0,\mp}^2=\{\theta\in S^2:J_\pm\theta=\mp\theta\text{ and }\theta\perp\langle\cdot,\cdot\rangle\},&
S_\pm^2=\{\theta\in S^2:J_\pm\theta=\pm\theta\},\\
\Lambda_{0,\mp}^2=\{\theta\in\Lambda^2:J_\pm\theta=\mp\theta\text{ and }\theta\perp\Omega\},&
\Lambda_\pm^2=\{\theta\in \Lambda^2:J_\pm\theta=\pm\theta\}.\gronko
\end{array}$$
The $\mathcal{O}$ module decomposition $V^*\otimes V^*=\Lambda^2\oplus S^2$ is an orthogonal direct sum of the 
alternating and symmetric $2$-tensors. We may further decompose:
\begin{eqnarray*}
&&\Lambda^2=\chi\oplus\Lambda_{0,+}^2\oplus\Lambda_-^2,\ 
S^2=\pone\oplus S^2_{0,+}\oplus S^2_-\quad\text{(complex setting)},\\
&&\Lambda^2=\chi\oplus\Lambda_{0,-}^2\oplus\Lambda_+^2,\ 
S^2=\pone\oplus S^2_{0,-}\oplus S^2_+\quad\text{(para-complex setting)}\,.
\end{eqnarray*}
The decompositions of $\Lambda^2$ and $S^2$ given above are into irreducible and inequivalent
$\mathcal{U}^\star$ modules. In the para-Hermitian setting, the modules are not irreducible 
if we replace $\mathcal{U}^\star$ by $\mathcal{U}$. In the Hermitian setting, these modules 
are still irreducible but not inequivalent if we replace $\mathcal{U}^\star$ by $\mathcal{U}$. 
Thus $\mathcal{U}^\star$ is the appropriate structure group for our purposes . 
The space of {\it algebraic K\"ahler--Weyl curvature tensors} is given by $\mathfrak{K}_{\mathfrak{W}}\subset\otimes^4V^*$ is defined by 
imposing the symmetries of Equation~(\ref{eqn-1.a}) and
Equation~(\ref{eqn-1.d}). There is \cite{B12,GN11c} an orthogonal direct sum decomposition 
of $\mathfrak{K}_{\mathfrak{W}}$ into inequivalent and
irreducible
$\mathcal{U}^\star$-modules:
$$\mathfrak{K}_{\mathfrak{W}}=
\left\{\begin{array}{l}
W_1\oplus W_2\oplus W_3\oplus L_{0,+}^2\oplus L_-^2\text{ (complex setting)}\\
W_1\oplus W_2\oplus W_3\oplus L_{0,-}^2\oplus L_+^2\text{ (para-complex setting)}
\end{array}\right\}\,.$$
We have  $W_3=\ker(\rho)\cap\mathfrak{K}_{\mathfrak{W}}$ and, furthermore, 
that $\rho_s$ and $\rho_a$ define $\mathcal{U}^\star$
module isomorphisms \cite{B12,GN12,GN11c}:
$$\operatorname{Tr}(\rho):W_1\mapright{\approx}\pone,\quad
  \rho_s:W_2\mapright{\approx}S_{0,\mp}^2,\quad
  \rho_a:L_{0,\mp}^2\mapright{\approx}\Lambda_{0,\mp}^2,\quad
  \rho_a:L_\pm^2\mapright{\approx}\Lambda_\pm^2\,.
$$
We note for further reference that in dimension $4$ we have:
$$\begin{array}{lll}
\dim\{W_1\}=1,&
\dim\{W_2\}=3,&
\dim\{W_3\}=5,\\
\dim\{\Lambda_{0,\mp}^2\}=3,&
\dim\{L_\pm^2\}=2\,.
\end{array}$$

\subsection{Lie groups}
Let $G$ be a $4$-dimensional Lie group which is equipped with an integrable left invariant 
complex structure (resp. para-complex structure) and a left invariant
pseudo-Hermitian metric (resp. para-Hermitian metric). Then the associated Lie algebra $\mathfrak{g}$ 
is equipped with an almost complex structure (resp. para-complex structure) with
vanishing Nijenhuis tensor and a pseudo-Hermitian (resp. para-Hermitian) inner product
$\langle\cdot,\cdot\rangle$. Conversely, given
$(\mathfrak{g},\langle\cdot,\cdot\rangle,J_\pm)$ where $\mathfrak{g}=(\mathbb{R}^4,[\cdot,\cdot])$ 
is a 4-dimensional Lie algebra equipped with an integrable pseudo-Hermitian ($-$) or para-Hermitian ($+$)
structure, there is a unique simply connected Lie group $G$ with Lie algebra $\mathfrak{g}$
so $J_\pm$ induces a left invariant integrable complex  (resp. para-complex) structure
on $G$ and $\langle\cdot,\cdot\rangle$
induces a left-invariant pseudo-Hermitian metric (resp. para-Hermitian metric) on $G$. 
Thus we can work in the algebraic context henceforth. 
Fix a pseudo-Hermitian ($-$) or para-Hermitian ($+$) vector space $(V,\langle\cdot,\cdot\rangle,J_\pm)$. 
Given $\Xi\in\Lambda_{0,+}^2\oplus\Lambda_-^2$ or $\Xi\in\Lambda_{0,-}^2\oplus\Lambda_+^2$, we shall
try to define a bracket $[\cdot,\cdot]$ so that $J_\pm$ is integrable and so that if $\nabla$ is the associated
K\"ahler--Weyl connection, then $\rho_a=\Xi$; by Equation~(\ref{eqn-1.b}) and Theorem~\ref{thm-1.2} one has:
$$\rho_a=-2d\phi=\left\{\begin{array}{ll}\phantom{-}dJ_-\delta\Omega\text{ if }J_-\text{ is complex}\\
-dJ_+\delta\Omega\text{ if }J_+\text{ is para-complex}\end{array}\right\}\,.
$$

The following theorem is the fundamental result of this paper 
(we refer to \cite{B12} for a
survey of other results concerning geometric realizability):
\begin{theorem}\label{thm-1.3}
\ \begin{enumerate}
\item Let $(V,\langle\cdot,\cdot\rangle,J_-)$ be a $4$-dimensional Hermitian
vector space of signature $(0,4)$. Then
every element of $\Lambda_{0,+}^2\oplus\Lambda_-^2$ 
is realizable by a $4$-dimensional Hermitian Lie group.
\item Let $(V,\langle\cdot,\cdot\rangle,J_+)$ be a $4$-dimensional para-Hermitian
vector space of signature $(2,2)$. Then
every element of $\Lambda_{0,-}^2\oplus\Lambda_+^2$ 
is realizable by a $4$-dimensional para-Hermitian Lie group.
\end{enumerate}\end{theorem}

\begin{remark}
\rm The corresponding geometrical realization question without the 
assumption of homogeneity was established
previously in \cite{GN12,GN11c}; the question at hand of providing homogeneous examples realizing all such tensors
$\Xi$ was posed to us by Prof. Alekseevsky and we are grateful to him for the suggestion. Related questions
have been examined previously. See, for example, the discussion in
\cite{K97} of homogeneous Einstein--Weyl structures on symmetric spaces, the discussion in \cite{ET00} of (complex)
$3$-dimensional homogeneous metrics which admit Einstein--Weyl connections, and the discussion in \cite{ET97} dealing
with (among other matters) the $4$-dimensional Einstein--Weyl equations in the homogeneous setting.\end{remark}

Here is a brief outline of the remainder of this paper. In Section~\ref{sect-2}, 
we recall some of the geometry of complex and para-complex Lie algebras and in Section~\ref{sect-3}, we establish 
Theorem~\ref{thm-1.3}.

\section{A review of complex and para-complex geometry}\label{sect-2}

\subsection{The action of the unitary group on $\Lambda_{0,\mp}^2\oplus\Lambda_\pm^2$}
If $G$ is a Lie group, let $G_0$ be the connected component of the identity. 
The following is a useful fact; we omit the proof as it is an
entirely elementary computation:

\begin{lemma}\label{lem-2.1}
Let
 $(V,\langle\cdot,\cdot\rangle,J_\pm)$ be a pseudo-Hermitian ($-$) or a
para-Hermitian ($+$) vector space. 
The natural action of the unitary group $\mathcal{U}$ on
 $\Lambda_{0,\mp}^2\oplus\Lambda_\pm^2$ defines a surjective group homomorphism $\pi$ from 
 $\mathcal{U}_0$ to $\mathcal{O}_0(\Lambda_{0,\mp}^2)\oplus\mathcal{O}_0(\Lambda_\pm^2)$.
 \end{lemma}

\subsection{Hermitian signature $(0,4)$}\label{sect-2.2}
Let $(V,\langle\cdot,\cdot\rangle,J_-)$ be a 4-dimensional Hermitian vector space of signature $(0,4)$. 
Let $\{e_1,e_2,e_3,e_4\}$ be an orthonormal basis for $V$ and let $\{e^1,e^2,e^3,e^4\}$ be
the associated dual basis for $V^*$. We normalize the choice
so that the
complex structure $J_-$ and diagonal inner products are given by:
\begin{eqnarray*}
&&J_-e_1=e_2,\quad J_-e_2=-e_1,\quad J_-e_3=e_4,\quad J_-e_4=-e_3,\\
&&\langle e_1,e_1\rangle=\langle e_2,e_2\rangle=\langle e_3,e_3\rangle=\langle e_4,e_4\rangle=2,\\
&&\langle e^1,e^1\rangle=\langle e^2,e^2\rangle=\langle e^3,e^3\rangle=\langle e^4,e^4\rangle=\textstyle\frac12.
\end{eqnarray*}
We define  a complex basis $\{Z_1,Z_2,\bar Z_1,\bar Z_2\}$ for $V_{\mathbb{C}}:=V\otimes_{\mathbb{R}}\mathbb{C}$
and the corresponding complex dual basis $\{Z^1,Z^2,\bar Z^1,\bar Z^2\}$ for the complex dual space $V_{\mathbb{C}}^*$ by setting:
$$\begin{array}{ll}
Z_1:=\frac12(e_1-\sqrt{-1}e_2),&Z_2:=\frac12(e_3-\sqrt{-1}e_4),\\
\bar Z_1:=\frac12(e_1+\sqrt{-1}e_2),&\bar Z_2:=\frac12(e_3+\sqrt{-1}e_4),\vphantom{\vrule height 11pt}\\
Z^1:=\phantom{\frac12}(e^1+\sqrt{-1}e^2),&Z^2:=\phantom{\frac12}(e^3+\sqrt{-1}e^4),\vphantom{\vrule height 11pt}\\
\bar Z^1:=\phantom{\frac12}(e^1-\sqrt{-1}e^2),&\bar Z^2:=\phantom{\frac12}(e^3-\sqrt{-1}e^4).\vphantom{\vrule height 11pt}
\end{array}$$
Then we have 
\medbreak\qquad$\langle Z_1,\bar Z_1\rangle=1$, $\langle Z^1,\bar Z^1\rangle=1$,
$\langle Z_2,\bar Z_2\rangle=1$, $\langle Z^2,\bar Z^2\rangle=1$,
\smallbreak\qquad
$J_-Z_1=\sqrt{-1}Z_1,\qquad\qquad\quad J_-\bar Z_1=-\sqrt{-1}\bar Z_1,$
\smallbreak\qquad
$J_-Z_2=\sqrt{-1}Z_2,\qquad\qquad\quad J_-\bar Z_2=-\sqrt{-1}\bar Z_2,$
\smallbreak\qquad
$J_-Z^1=\sqrt{-1}Z^1,\qquad\qquad\quad  J_-\bar Z^1=-\sqrt{-1}\bar Z^1,$
\smallbreak\qquad
$J_-Z^2=\sqrt{-1}Z^2,\qquad\qquad\quad  J_-\bar Z^2=-\sqrt{-1}\bar Z^2$.
\medbreak\noindent
Because $\Omega(Z_j,\bar Z_j)=\langle Z_j,J_-\bar Z_j\rangle=-\sqrt{-1}$, the K\"ahler form is given by:
$$
\Omega=-\sqrt{-1}\{Z^1\wedge\bar Z^1+ Z^2\wedge\bar Z^2\}\,.
$$
Let $[\cdot,\cdot]$ be a Lie bracket on $V_{\mathbb{C}}$. Then
 $[\cdot,\cdot]$ is integrable if and only if
$$
[Z_1,Z_2]\in\operatorname{Span}\{Z_1,Z_2\}
$$
and $[\cdot,\cdot]$ arises from an underlying real bracket on $V$
if and only if 
$$
\overline{[x,y]}=[\bar x,\bar y]\text{ for all }x,y\in V_{\mathbb{C}}\,.
$$
We define a basis
$\{\theta_1,\theta_2,\theta_3\}$ for $\Lambda_{0,+}^2$ and a basis $\{\theta_4,\theta_5\}$ for $\Lambda_-^2$ by setting:
$$\begin{array}{ll}
\theta_1:=\sqrt{-1}(Z^1\wedge \bar Z^1
      - Z^2\wedge\bar Z^2),&
\theta_2:=Z^1\wedge \bar Z^2+\bar Z^1\wedge Z^2,\\
\theta_3:=\sqrt{-1}(Z^1\wedge \bar Z^2-\bar Z^1\wedge Z^2),\\
\theta_4:=Z^1\wedge Z^2+\bar Z^1\wedge \bar Z^2,&
\theta_5:=\sqrt{-1}(Z^1\wedge Z^2-\bar Z^1\wedge \bar Z^2)\,.
\end{array}$$
The collection $\{\theta_1,...,\theta_5\}$ is an orthogonal set with the diagonal inner products given by:
$$\begin{array}{lllll}
\langle\theta_1,\theta_1\rangle=2,&
\langle\theta_2,\theta_2\rangle=2,&
\langle\theta_3,\theta_3\rangle=2,&
\langle\theta_4,\theta_4\rangle=2,&
\langle\theta_5,\theta_5\rangle=2\,.
\end{array}$$
Consequently, $\Lambda_{0,+}^2$ has signature $(0,3)$ and
$\Lambda_-^2$ has signature $(0,2)$. Let
$$\Theta^{x,y}:=\{\Xi\in\Lambda_{0,+}^2\oplus\Lambda_-^2:|\Xi_{0,+}|^2=x\text{ and }|\Xi_-|^2=y\}\,.$$

\begin{lemma}\label{lem-2.2}
If $(V,\langle\cdot,\cdot\rangle,J_-)$ is a Hermitian inner product space of signature $(0,4)$, then $\{\Theta^{x,y}\}_{x\ge0,y\ge0}$ 
are the orbits of $\mathcal{U}$ acting on $\Lambda_{0,+}^2\oplus\Lambda_-^2$.\end{lemma}

\begin{proof} The sets $\Theta^{x,y}$ are the product of a sphere of radius $x$ in $\Lambda_{0,+}^2$ and a
sphere of radius $y$ in $\Lambda_-^2$ and thus represent the orbits of 
$\mathcal{O}_0(\Lambda_{0,+}^2)\oplus\mathcal{O}_0(\Lambda_-^2)$ acting
on $\Lambda_{0,+}^2\oplus\Lambda_-^2$. Consequently, by Lemma~\ref{lem-2.1}, $\mathcal{U}$ acts transitively on these sets.
On the other hand, since $\mathcal{U}$ acts orthogonally, $\mathcal{U}$ preserves these sets. \end{proof}

\subsection{Para-Hermitian signature $(2,2)$}\label{sect-2.3}
Let $(V,\langle\cdot,\cdot\rangle,J_+)$ be a para-Hermitian vector space of signature $(2,2)$. 
Choose a basis
$\{e_1,e_2,e_3,e_4\}$ for $V$ so that the basis is hyperbolic and so that $J_+$ is diagonalized:
\begin{eqnarray*}
&&\langle e_1,e_3\rangle=\langle e_2,e_4\rangle=1,\\
&&J_+e_1=e_1,\quad J_+e_2=e_2,\quad J_+e_3=-e_3,\quad J_+e_4=-e_4\,.
\end{eqnarray*}
We then have that
$$\Omega=-e^1\wedge e^3-e^2\wedge e^4\,.$$
A Lie bracket on $V$ is integrable
if and only if 
$$
[e_1,e_2]\in\operatorname{Span}\{e_1,e_2\}\quad\text{and}\quad[e_3,e_4]\in\operatorname{Span}\{e_3,e_4\}\,.
$$ 
We define an orthogonal basis $\{\theta_1,\theta_2,\theta_3\}$ for $\Lambda_{0,-}^2$ and 
an orthogonal basis $\{\theta_4,\theta_5\}$ for $\Lambda_+^2$ by setting:
$$\begin{array}{lll}
\theta_1:=e^1\wedge e^3-e^2\wedge e^4,&\theta_2:=e^1\wedge e^4+e^2\wedge e^3,&
\theta_3:=e^1\wedge e^4-e^2\wedge e^3,\\
\theta_4:=e^1\wedge e^2+e^3\wedge e^4,&\theta_5:=e^1\wedge e^2-e^3\wedge e^4.
\end{array}$$
The diagonal inner products are given by:
$$\begin{array}{lllll}
\langle\theta_1,\theta_1\rangle=-2,&
\langle\theta_2,\theta_2\rangle=-2,&
\langle\theta_3,\theta_3\rangle=2,\\
\langle\theta_4,\theta_4\rangle=\phantom{-}2,&
\langle\theta_5,\theta_5\rangle=-2.
\end{array}$$
Thus $\Lambda_{0,-}^2$ has signature $(2,1)$ and $\Lambda_+^2$ has signature $(1,1)$.
\begin{lemma}\label{lem-2.3}
Every orbit of the action of $\mathcal{U}$ on 
$\Lambda_{0,-}^2\oplus\Lambda_+^2$ contains a representative perpendicular to $\theta_1$.
\end{lemma}

\begin{proof}
$\{\theta_1,\theta_2,\theta_3\}$
is an orthogonal basis for $\Lambda_{0,-}^2$ where $\{\theta_1,\theta_2\}$ are timelike and $\theta_3$ is spacelike. Lemma~\ref{lem-2.3} follows
from Lemma~\ref{lem-2.1} since $\pi(\mathcal{U}_0)$ contains $\mathcal{O}_0(\Lambda_{0,-}^2)$ and the corresponding
assertion holds for this group.
\end{proof}

\section{The proof of Theorem~\ref{thm-1.3}}\label{sect-3}

In Section~\ref{sect-3.1}, we discuss the Hodge operator. In Section~\ref{sect-3.2}, we discuss a specific Lie algebra
which will be used in Section~\ref{sect-3.3} to prove Theorem~\ref{thm-1.3}~(2) and which
will be used  in Section~\ref{sect-3.4} to prove Theorem~\ref{thm-1.3}~(1).

\subsection{The Hodge $\star$ operator}\label{sect-3.1}
Let $\{\Psi_1,\Psi_2,\Psi_3,\Psi_4\}$ be a basis for $\mathbb{C}^4$ and let $\{\Psi^1,\Psi^2,\Psi^3,\Psi^4\}$ be 
the corresponding dual basis for the dual vector space. Take a hyperbolic metric whose non-zero components
are defined by:
$$\langle\Psi_1,\Psi_3\rangle=\langle\Psi_2,\Psi_4\rangle=\langle \Psi^1,\Psi^3\rangle=\langle \Psi^2,\Psi^4\rangle=1\,.$$
This is a convenient notation as it is consistent with previous sections:
\begin{enumerate} 
\item For Section~\ref{sect-2.2}, set $\Psi_1=Z_1$, $\Psi_2=Z_2$, $\Psi_3=\bar Z_1$, and
$\Psi_4=\bar Z_2$. 
\item For Section~\ref{sect-2.3}, set $\Psi_1=e_1$, $\Psi_2=e_2$, $\Psi_3=e_3$, and $\Psi_4=e_4$.
\end{enumerate}

Let $\star$ be the {\it Hodge operator}, let
$\dvol =\Psi^1\wedge \Psi^3\wedge \Psi^2\wedge \Psi^4$
be the volume form, and let $\delta$ be the co-derivative. We use the identity
$$\omega_1\wedge\star\omega_2=g(\omega_1,\omega_2)\dvol$$ to compute:
$$
\begin{array}{llll}
\star \Psi^1=-\Psi^1\wedge \Psi^2\wedge \Psi^4,&
\star \Psi^2=\phantom{-}\Psi^1\wedge \Psi^2\wedge \Psi^3,&
\star \Psi^3=-\Psi^2\wedge \Psi^3\wedge \Psi^4,\\
\star \Psi^4=\phantom{-}\Psi^1\wedge \Psi^3\wedge \Psi^4,&
\star \Psi^1\wedge \Psi^2=-\Psi^1\wedge \Psi^2,&
\star \Psi^1\wedge \Psi^3=-\Psi^2\wedge \Psi^4,\\
\star \Psi^1\wedge \Psi^4=\phantom{-}\Psi^1\wedge \Psi^4,&
\star \Psi^2\wedge \Psi^3=\phantom{-}\Psi^2\wedge \Psi^3,&
\star \Psi^2\wedge \Psi^4=-\Psi^1\wedge \Psi^3,\\
\star \Psi^3\wedge \Psi^4=-\Psi^3\wedge \Psi^4,&
\star \Psi^1\wedge \Psi^2\wedge \Psi^3=-\Psi^2,&
\star \Psi^1\wedge \Psi^2\wedge \Psi^4=\phantom{-}\Psi^1,\\
\star \Psi^1\wedge \Psi^3\wedge \Psi^4=-\Psi^4,&
\star \Psi^2\wedge \Psi^3\wedge \Psi^4=\phantom{-}\Psi^3.
\end{array}$$

\subsection{An example:}\label{sect-3.2}
 We define a complex Lie algebra by setting:$$\begin{array}{lll}
\ [\Psi_1,\Psi_2]=\varepsilon_1\Psi_1,&\ [\Psi_1,\Psi_4]=\alpha_3\Psi_1,&\ [\Psi_2,\Psi_3]=-\tilde\alpha_3\Psi_3,\\
\ [\Psi_2,\Psi_4]=\alpha_2\Psi_1-\tilde\alpha_2\Psi_3,&\ [\Psi_3,\Psi_4]=\tilde\varepsilon_1\Psi_3.
\end{array}$$
We verify that the Jacobi identity is satisfied:
\begin{eqnarray*}
&&[[\Psi_1,\Psi_2],\Psi_3]+[[\Psi_2,\Psi_3],\Psi_1]+[[\Psi_3,\Psi_1],\Psi_2]\\&&\quad
=\varepsilon_1[\Psi_1,\Psi_3]-\tilde\alpha_3[\Psi_3,\Psi_1]+0=0+0+0=0,\\
&&[[\Psi_1,\Psi_2],\Psi_4]+[[\Psi_2,\Psi_4],\Psi_1]+[[\Psi_4,\Psi_1],\Psi_2]\\&&\quad
=\varepsilon_1[\Psi_1,\Psi_4]+[\alpha_2\Psi_1-\tilde\alpha_2\Psi_3,\Psi_1]-\alpha_3[\Psi_1,\Psi_2]\\&&\quad
=\varepsilon_1\alpha_3\Psi_1+0-\alpha_3\varepsilon_1\Psi_1=0,\\
&&[[\Psi_1,\Psi_3],\Psi_4]+[[\Psi_3,\Psi_4],\Psi_1]+[[\Psi_4,\Psi_1],\Psi_3]\\&&\quad
=0+\tilde\varepsilon_1[\Psi_3,\Psi_1]-\alpha_3[\Psi_1,\Psi_3]=0+0+0=0,\\
&&[[\Psi_2,\Psi_3],\Psi_4]+[[\Psi_3,\Psi_4],\Psi_2]+[[\Psi_4,\Psi_2],\Psi_3]\\&&\quad
=-\tilde\alpha_3[\Psi_3,\Psi_4]+\tilde\varepsilon_1[\Psi_3,\Psi_2]-[\alpha_2\Psi_1-\tilde\alpha_2\Psi_3,\Psi_3]\\&&\quad
=-\tilde\alpha_3\tilde\varepsilon_1\Psi_3+\tilde\varepsilon_1\tilde\alpha_3\Psi_3+0=0\,.
\end{eqnarray*}
We define a para-complex structure $J_+$ setting:
$$\begin{array}{llll}
J_+\Psi^1=\Psi^1,& J_+\Psi^2=\Psi^2,&J_+\Psi^3=-\Psi^3,&J_+\Psi^4=-\Psi^4.
\end{array}$$
 We then have
 $$
 \Omega_+:=-(\Psi^1\wedge \Psi^3+\Psi^2\wedge \Psi^4).
$$
We use the formula $d\Psi^i(\Psi_j,\Psi_k)=-\Psi^i([\Psi_j,\Psi_k])$ to compute:
$$\begin{array}{ll}
d\Psi^1=-\epsilon_1\Psi^1\wedge\Psi^2-\alpha_3\Psi^1\wedge\Psi^4-\alpha_2\Psi^2\wedge\Psi^4,&d\Psi^2=0,\\
d\Psi^3=\phantom{-}\tilde\alpha_3\Psi^2\wedge\Psi^3+\tilde\alpha_2\Psi^2\wedge\Psi^4-\tilde\varepsilon_1\Psi^3\wedge\Psi^4,&
d\Psi^4=0.
\end{array}$$
Since $\delta=-\star d\star$ and $\star\Omega=-\Omega$, we have:
\medbreak\quad
$\delta\Omega_+=-\star d\star\Omega_+=-\star d(\Psi^1\wedge \Psi^3+\Psi^2\wedge \Psi^4)$
\smallbreak\qquad\quad
$=\star\{(\varepsilon_1\Psi^1\wedge \Psi^2+\alpha_3\Psi^1\wedge \Psi^4+\alpha_2\Psi^2\wedge \Psi^4)\wedge \Psi^3\}$
\smallbreak\qquad\quad
$+\star\{\Psi^1\wedge(\tilde\alpha_3\Psi^2\wedge \Psi^3+\tilde\alpha_2\Psi^2\wedge \Psi^4
-\tilde\varepsilon_1\Psi^3\wedge \Psi^4)\}$
\smallbreak\qquad\quad
$=\tilde\alpha_2\Psi^1-(\varepsilon_1+\tilde\alpha_3)\Psi^2 -\alpha_2\Psi^3+(\tilde\varepsilon_1+\alpha_3)\Psi^4$,
\medbreak\quad
$dJ_+\delta\Omega_+=d\{\tilde\alpha_2\Psi^1-(\varepsilon_1+\tilde\alpha_3)\Psi^2
+\alpha_2\Psi^3-(\tilde\varepsilon_1+\alpha_3)\Psi^4\}$
\smallbreak\qquad\quad
$=-\tilde\alpha_2\varepsilon_1\Psi^1\wedge \Psi^2-\tilde\alpha_2\alpha_3\Psi^1\wedge \Psi^4
-\tilde\alpha_2\alpha_2\Psi^2\wedge \Psi^4$
\smallbreak\qquad\qquad
$+\alpha_2\tilde\alpha_3\Psi^2\wedge \Psi^3+\alpha_2\tilde\alpha_2\Psi^2\wedge \Psi^4
-\alpha_2\tilde\varepsilon_1\Psi^3\wedge \Psi^4$
\medbreak\noindent which by Theorem~\ref{thm-1.2} yields
\begin{equation}\label{eqn-3.a}
\rho_a=\tilde\alpha_2\varepsilon_1\Psi^1\wedge \Psi^2+\tilde\alpha_2\alpha_3\Psi^1\wedge\Psi^4
-\alpha_2\tilde\alpha_3\Psi^2\wedge\Psi^3
+\alpha_2\tilde\varepsilon_1\Psi^3\wedge\Psi^4.
\end{equation}

\subsection{The proof of Theorem~\ref{thm-1.3}~(2)}\label{sect-3.3}

We now deal with the para-Hermitian setting. Let $\Xi\in\Lambda_{0,-}^2\oplus\Lambda_+^2$. By Lemma~\ref{lem-2.3}, 
we may assume that the coefficient of
$\theta_1=e^1\wedge e^3-e^2\wedge e^4$ in $\Xi$ vanishes, i.e. 
$$
\Xi=\mu_{12}e^1\wedge e^2+\mu_{14}e^1\wedge e^4+\mu_{23}e^2\wedge e^3+\mu_{34}e^3\wedge e^4\,.
$$
We must show that $\Xi$ is geometrically realizable by a $4$-dimensional para-Hermitian Lie group.
We consider the Lie algebra of Section~\ref{sect-3.2} where the parameters 
$\{\varepsilon_1,\tilde\varepsilon_1,\alpha_2,\tilde\alpha_2,\alpha_3,\tilde\alpha_3\}$ are real
and where we set $\Psi_i=e_i$; this Lie algebra is modeled on $A_{2,2}\oplus A_{2,2}$ in 
the classification of \cite{P76} for generic values of the parameters. We apply Equation~(\ref{eqn-3.a}).
We set $\alpha_2=\tilde\alpha_2=1$. The remaining parameters are then determined; we complete the proof by taking:
\medbreak\qquad\qquad
$\varepsilon_1=\mu_{12}$, $\alpha_3=\mu_{14}$, 
$\tilde\alpha_3=-\mu_{23}$, $\tilde\varepsilon_1=\mu_{34}$.\hfill\qed

\subsection{The proof of Theorem~\ref{thm-1.3}~(1)}\label{sect-3.4}
The question of realizability is invariant under the action of the structure
group $\mathcal{U}$. Thus by Lemma~\ref{lem-2.2}, only the norms of $|\Xi_{0,+}|$ and $|\Xi_-|$ are relevant
in establishing Theorem~\ref{thm-1.3}~(1). Again, we use the Lie algebra of Section~\ref{sect-3.2}.
We set $\Psi_1=Z_1$, $\Psi_2=Z_2$, $\Psi_3=\bar Z_1$, $\Psi_4=\bar Z_2$, and take $\tilde\varepsilon_1=\bar\varepsilon_1$,
$\tilde\alpha_2=\bar\alpha_2$, and $\tilde\alpha_3=\bar\alpha_3$ to define an underlying real Algebra
which is modeled on $A_{4,12}$ in the classification of 
\cite{P76} for generic values of the parameters; see also related work in \cite{O03,O04,RS10,S09}. We set
$J_-:=\sqrt{-1}J_+$. We then have $\Omega_-=\sqrt{-1}\Omega_+$ so
\begin{eqnarray*}
&&\phi_-=-\textstyle\frac12J_-\delta\star\Omega_-=\textstyle\frac12J_+\delta\star\Omega_+=\phi_+,\\
&&\rho_a=\bar\alpha_2\varepsilon_1Z^1\wedge Z^2+\bar\alpha_2\alpha_3Z^1\wedge \bar Z^2
-\alpha_2\bar\alpha_3Z^2\wedge \bar Z^1
+\alpha_2\bar\varepsilon_1\bar Z^1\wedge \bar Z^2,\\
&&|\Xi_{0,+}|^2=2|\alpha_2|^2|\alpha_3|^2,\qquad
    |\Xi_-|^2=2|\alpha_2|^2|\varepsilon_1|^2\,.
\end{eqnarray*}
If we set $\alpha_2=1$, \quad $\alpha_3=\sqrt{x/2}$, and $\varepsilon_1=\sqrt{y/2}$, then
$\rho_a\in\Theta^{x,y}$ as desired.\hfill\qed

\section*{Acknowledgments}
Research of the authors partially supported by project MTM2009-07756 (Spain) and by INCITE09 207 151 PR (Spain).

\end{document}